\documentclass[11pt,twoside,english]{article}
\usepackage{mathptmx}
\usepackage[T1]{fontenc}
\usepackage[utf8]{inputenc}
\usepackage[a4paper]{geometry}
\usepackage{color}
\usepackage{babel}
\usepackage{amsmath}
\usepackage{amsthm}
\usepackage{amssymb}
\usepackage{esint}
\usepackage[unicode=true,pdfusetitle,
 bookmarks=true,bookmarksnumbered=false,bookmarksopen=false,
 breaklinks=false,pdfborder={0 0 1},backref=false,colorlinks=true]
 {hyperref}

\makeatletter
\numberwithin{equation}{section}
\numberwithin{figure}{section}
\theoremstyle{plain}
\newtheorem{thm}{\protect\theoremname}
\theoremstyle{plain}
\newtheorem{cor}[thm]{\protect\corollaryname}
\theoremstyle{definition}
\newtheorem{defn}[thm]{\protect\definitionname}
\theoremstyle{plain}
\newtheorem{prop}[thm]{\protect\propositionname}
\theoremstyle{plain}
\newtheorem{lem}[thm]{\protect\lemmaname}
\theoremstyle{remark}
\newtheorem{rem}[thm]{\protect\remarkname}

\allowdisplaybreaks[4]
\date{}

\makeatother

\providecommand{\corollaryname}{Corollary}
\providecommand{\definitionname}{Definition}
\providecommand{\lemmaname}{Lemma}
\providecommand{\propositionname}{Proposition}
\providecommand{\remarkname}{Remark}
\providecommand{\theoremname}{Theorem}

\begin{document}
\global\long\def\divg{{\rm div}\,}%

\global\long\def\curl{{\rm curl}\,}%

\global\long\def\rt{\mathbb{R}^{3}}%

\global\long\def\rn{\mathbb{R}^{n}}%

\global\long\def\dir{\mathcal{E}}%

\title{Markov semi-groups generated by elliptic operators with divergence-free
drift}
\author{Zhongmin Qian\thanks{Research supported partly by an ERC grant Esig ID 291244. Mathematical
Institute, University of Oxford, OX2 6GG, England. Email: qianz@maths.ox.ac.uk} \ and\ Guangyu Xi\thanks{This work was supported by Engineering and Physical Sciences Research
Council {[}EP/L015811/1{]}. Department of mathematics, University
of Maryland, College Park. 20742, US. Email: gxi@umd.edu}}
\maketitle
\begin{abstract}
In this paper we construct a conservative Markov semi-group with generator
$L=\Delta+b\cdot\nabla$ on $\rn$, where $b$ is a divergence-free
vector field which belongs to $L^{2}\cap L^{p}$ with $\frac{n}{2}<p$.
The research is motivated by the question of understanding the blow-up
solutions of the fluid dynamic equations, which attracts a lot of
attention in recent years. 
\end{abstract}

\medskip


\section{Introduction and the main result}

In fluid dynamics, the velocity $u(t,x)$ of fluid particles is described
by, in the case of incompressible fluids, the Navier-Stokes equations
\begin{equation}
\partial_{t}u+u\cdot\nabla u=\nu\triangle u-\nabla p,\quad\nabla\cdot u=0,\label{eq:nseq1}
\end{equation}
in a domain of Euclidean space $\mathbb{R}^{3}$, subject to certain
initial and boundary conditions. Here $p(t,x)$ is the pressure which
is uniquely determined by $u(t,x)$ up to a constant at every $t$,
and it solves the Poisson equation $\Delta p=-\divg(u\cdot\nabla u)$.
Hence $p(t,x)$ is a non-linear and non-local term in the Naiver-Stokes
equations. 

The first equation in (\ref{eq:nseq1}) can be written as a parabolic
type equation:
\begin{equation}
\left(\partial_{t}-\nu\triangle+u\cdot\nabla\right)u=-\nabla p,\label{eq:pbe1}
\end{equation}
which however possesses no much common features as (local) parabolic
equations. But nevertheless the theory of parabolic equations is helpful
in the analysis of the Navier-Stokes equations. It is a matter of
fact that many quantities related to fluid flows such as the vorticity
and the stress tensor fields also satisfy the same kind of parabolic
evolution equations with the principal parabolic operator $\partial_{t}-L$,
where $L=\nu\triangle-u\cdot\nabla$. The operator $L$ is time non-homogeneous
since $u$ depends on $t$, and its formal adjoint $L^{*}=\nu\triangle+u\cdot\nabla$
is the infinitesimal generator of a diffusion process, called Taylor's
diffusion, which models the fluid flows in terms of Brownian particles.
Taylor's diffusion solves formally the following stochastic differential
equation
\begin{equation}
dX_{t}=u(t,X_{t})+\sqrt{2\nu}dW_{t},\label{eq:diff1}
\end{equation}
where $W_{t}$ is the Brownian motion. Taylor's diffusion has been
an important tool in the study of turbulent flows and in the development
of numerical simulations to the solutions of the Navier-Stokes equations
(such as vortex methods). For the Navier-Stokes equations, only global
weak solutions have been constructed in general, and knowledge of
weak solutions is still limited. There is a vast literature addressing
the regularity of weak solutions, see e.g. \cite{seregin2015lecture,temam1984navier}.
Leray's weak solution $u(t,x)$ satisfies the energy balance equation,
which implies that 
\begin{equation}
u(t,x)\in L^{2}(0,T;H^{1})\cap L^{\infty}(0,T;L^{2}).\label{eq: leray weak solution}
\end{equation}
For the most interesting case where dimension is three, this regularity
only implies that $u(t,x)\in L^{2}(0,T;L^{6}(\mathbb{R}^{3}))\cap L^{\infty}(0,T;L^{2}(\rt))$
and the classical parabolic regularity theory fails to apply.

Consider the following parabolic equation 
\begin{equation}
\partial_{t}u(t,x)-\sum_{i,j=1}^{n}\partial_{x_{i}}(a_{ij}(t,x)\partial_{x_{j}}u(t,x))+\sum_{i=1}^{n}b_{i}(t,x)\partial_{x_{i}}u(t,x)=0\label{eq: 1. problem equation}
\end{equation}
on $\rn$. A classical monograph on this class of parabolic equations
is \cite{ladyzhenskaia1988linear} by Ladyzhenskaya et al., in which
the existence of a unique Hölder continuous weak solution $u$ is
proved under the conditions that $a$ is uniform elliptic and $b\in L^{l}(0,T;L^{q}(\rn))$
with $\frac{2}{l}+\frac{n}{q}\leq1$, $l\neq\infty$. It is not known
in general whether a Leray's weak solution has this regularity or
not. Actually we know that when $n=3$, (\ref{eq: leray weak solution})
implies that $\frac{2}{l}+\frac{3}{q}=\frac{3}{2}$. This motivates
us to consider the cases that $\frac{2}{l}+\frac{n}{q}>1$ together
with the assumption that $b$ is divergence-free.

If $(a_{ij}$) and $(b_{i}$) are smooth, then there exists a unique
fundamental solution $\Gamma(t,x;\tau,\xi)$ associated with the Cauchy
initial problem (\ref{eq: 1. problem equation}). In \cite{qian2017parabolic},
the following Aronson type estimate in the time-inhomogeneous case
with a super-critical drift $b$ has been established, see also the
related estimates in \cite{aronson1968non,norris1991estimates,qian2016parabolic,zhang2004strong}.
\begin{thm}
\label{thm: 1 upper bound}Suppose $a=(a_{ij})$ and $b=(b_{i})$
are smooth which satisfy that $\lambda\vert\xi\vert^{2}\leq\sum_{i,j=1}^{n}a_{ij}\xi_{i}\xi_{j}\leq\frac{1}{\lambda}\vert\xi\vert^{2}$
and $\divg b=0$, and assume that $b\in L^{l}(0,T;L^{q}(\rn))$ for
some $n\geq3$, $l>1$, $q>\frac{n}{2}$ such that $1\leq\frac{2}{l}+\frac{n}{q}<2$.
If $\mu\equiv\frac{2}{2-\gamma+\frac{2}{l}}>1$ with $\gamma=\frac{2}{l}+\frac{n}{q}$,
then the fundamental solution has upper bound
\begin{equation}
\Gamma(t,x;\tau,\xi)\leq\begin{cases}
\frac{C_{1}}{(t-\tau)^{n/2}}\exp\left(-\frac{1}{C_{2}}\left(\frac{\vert x-\xi\vert^{2}}{t-\tau}\right)\right) & \frac{\vert x-\xi\vert^{\mu-2}}{(t-\tau)^{\mu-\nu-1}}<1\\
\frac{C_{1}}{(t-\tau)^{n/2}}\exp\left(-\frac{1}{C_{2}}\left(\frac{\vert x-\xi\vert^{\mu}}{(t-\tau)^{\nu}}\right)^{\frac{1}{\mu-1}}\right) & \frac{\vert x-\xi\vert^{\mu-2}}{(t-\tau)^{\mu-\nu-1}}\geq1
\end{cases}\label{eq: 3. upper bound-2}
\end{equation}
where $\nu=\frac{2-\gamma}{2-\gamma+\frac{2}{l}}$, $\Lambda=\Vert b\Vert_{L^{l}(0,T;L^{q}(\rn))}$,
$C_{1}=C_{1}(l,q,n,\lambda)$, $C_{2}=C_{2}(l,q,n,\lambda,\Lambda)$.
If $\mu=1$, so that $q=\infty$, then
\begin{equation}
\Gamma(t,x;\tau,\xi)\leq\frac{C_{1}}{(t-\tau)^{n/2}}\exp\left(-\frac{(C_{1}\Lambda(t-\tau)^{\nu}-\vert x-\xi\vert)^{2}}{4C_{1}(t-\tau)}\right).
\end{equation}
\end{thm}

The upper bound in Theorem \ref{thm: 1 upper bound} implies that
$\Gamma(t,x;\tau,\xi)$ decays exponentially in space variables, which
yields the pre-compactness of the family of probability measures defined
by $\varGamma$, in the sense that, the family of finite dimensional
distributions
\[
\prod_{i=1}^{n}\Gamma(t_{i},x_{i};t_{i-1},x_{i-1})\ dx_{0}\cdots dx_{n}
\]
for fixed $s\leq t_{0}<t_{1}<\cdots<t_{n}$ is relatively compact
under the topology of weak convergence in measure. The relative compactness
allows us to construct $\Gamma(t,x;\tau,\xi)$ for Borel measurable
$a$ and $b$ which satisfy that $\lambda\vert\xi\vert^{2}\leq\sum_{i,j=1}^{n}a_{ij}\xi_{i}\xi_{j}\leq\frac{1}{\lambda}\vert\xi\vert^{2}$,
$\divg b=0$, and $b\in L^{l}(0,T;L^{q}(\rn))$ with $\frac{2}{l}+\frac{n}{q}\in[1,2)$,
but the convergence is too weak to ensure the Chapman–Kolmogorov equation,
i.e. 

\[
\Gamma(t,x;\tau,\xi)=\int_{\rn}\Gamma(t,x;s,y)\Gamma(s,y;\tau,\xi)\ dy.
\]

Therefore, in this paper, we consider the time-homogeneous parabolic
equation 

\begin{equation}
\partial_{t}u(t,x)-\sum_{i,j=1}^{n}\partial_{x_{i}}(a_{ij}(x)\partial_{x_{j}}u(t,x))+\sum_{i=1}^{n}b_{i}(x)\partial_{x_{i}}u(t,x)=0\label{eq: problem equation 1}
\end{equation}
where $(a,b)$ satisfies that there exists a constant $\lambda>0$
such that 
\begin{equation}
\lambda\vert\xi\vert^{2}\leq\sum_{i,j=1}^{n}a_{ij}\xi_{i}\xi_{j}\leq\frac{1}{\lambda}\vert\xi\vert^{2}\tag*{E}\label{eq: 1. ellipticity}
\end{equation}
and $b$ is divergence-free 
\begin{equation}
\sum_{i=1}^{n}\partial_{x_{i}}b_{i}(x)=0.\tag*{S}\label{eq: 1. divergence free}
\end{equation}
It worth noting that without the divergence-free condition (\ref{eq: 1. divergence free}),
(\ref{eq: problem equation 1}) may not have a solution for super-critical
$b$ (see for example \cite[Example 1]{kinzebulatov2020} and \cite{zhang2020}).

The existence of weak solutions to (\ref{eq: problem equation 1})
under mild conditions on $b$ is easy to prove, while the uniqueness
is not trivial. Here we prove the uniqueness using the conservativeness
implied by Theorem \ref{thm: 1 upper bound}. To be more specific,
we construct a unique Markov semi-group $(P_{t})_{t\geq0}$ for $b\in L^{2}(\rn)\cap L^{q}(\rn)$,
$q>\frac{n}{2}$ and show that it has kernel $\Gamma$, which satisfies
the Chapman–Kolmogorov equation. Moreover, the semi-group is conservative,
i.e. 
\begin{equation}
P_{t}1(x)=\int_{\mathbb{R}^{n}}\Gamma(t,x,y)\ dy=1\label{eq:con1}
\end{equation}
for a.e. $x\in\rn$. Notice that in this case the corresponding bi-linear
form
\begin{equation}
\dir(u,v)=\int_{\rn}\left[\langle\nabla u,a\cdot\nabla v\rangle+(b\cdot\nabla u)v\right]\ dx\label{eq:bi-form}
\end{equation}
is not sectorial in general in the sense defined in \cite{fukushima2010dirichlet,ma2012introduction}.

We are now in a position to state the main result of this paper.
\begin{thm}
\label{thm: Main result}Suppose conditions (\ref{eq: 1. ellipticity}),
(\ref{eq: 1. divergence free}) hold and $b\in L^{2}(\rn)\cap L^{q}(\rn)$
for $q>\frac{n}{2}$. There is a unique conservative Markov semi-group
$(P_{t})_{t\geq0}$ on $L^{2}(\rn)$ associated with operator $\divg(a\cdot\nabla)-b\cdot\nabla$
and it has transition probability kernel $\Gamma(t,x,y)$ for $t>0$,
$x,y\in\rn$. Moreover, the uniqueness of weak solutions holds for
the Cauchy initial problem to (\ref{eq: problem equation 1}) and
the solution is given by 
\begin{equation}
u(t,x)=\int_{\rn}\Gamma(t,x,y)u_{0}(y)\ dy\label{eq: representation of the solution}
\end{equation}
for any initial data $u_{0}\in L^{2}(\rn)$.
\end{thm}

Here we assume that $b$ is divergence-free and $b\in L^{q}(\mathbb{R}^{n})$
for $q>\frac{n}{2}$ are primarily for proving the conservativeness
of the semi-group through Theorem \ref{thm: 1 upper bound}. The conservativeness
is crucial for proving the uniqueness (see also \cite{stannat1999nonsymmetric}).
When the dimension $n=3$, the condition of Theorem \ref{thm: Main result}
is satisfied if $b\in L^{2}(\rt)$.

In the past decades, there have been many researches on the semi-groups
for non-sectorial bi-linear forms (\ref{eq:bi-form}). The construction
of the semi-group is studied in, for example, \cite{kinzebulatov2020,kovalenko1991c_0,liskevich1996c_0,sobol2002lp,stannat1999nonsymmetric,trutnau2003skorokhod}.
The construction of the semi-groups in these papers all use approximation
arguments of different types. In \cite{stannat1999nonsymmetric},
Stannat shows the existence of sub-Markovian $C_{0}$-semi-group when
$b\in L_{loc}^{2}(\rn)$. Although $b\in L_{loc}^{2}(\rn)$ does not
guarantee uniqueness, Stannat proves the equivalence between uniqueness,
conservativeness and a characterization of the bi-linear form (see
\cite[Proposition 1.9 ]{stannat1999nonsymmetric}).

The proof of uniqueness and conservativeness is harder. One could
prove them using the heat kernel estimates \cite{Davies1992}. In
\cite{Zhikov-2006}, using the Aronson type heat kernel estimate,
Zhikov constructed the unique approximation semi-group to 
\begin{equation}
\partial_{t}u-\divg((A+B)\cdot\nabla u)=0,\label{eq: zhikov problem}
\end{equation}
for periodic $B\in L^{\infty}(\rn)$, $\divg B\in L_{loc}^{2}(\rn)$
and $\sup_{r\geq1}\frac{1}{r^{n}}\Vert B\Vert_{L^{n}(B(0,r))}^{n}<\infty$.
Here $A$ is symmetric matrix-valued and $B$ is anti-symmetric matrix-valued.
It is easy to see that such problems are equivalent to (\ref{eq: problem equation 1})
with divergence-free $b$ if we set $a=A$ and $b=-\divg B$. In \cite{liskevich2003estimates},
Liskevich and Sobol further proved the heat kernel estimate for the
semi-groups under additional functional conditions on the bi-linear
form using the idea developed in \cite{davies1987explicit}. To prove the conservativeness, another idea is to assume growth conditions on the coefficient,
analyze locally and then take limit to the whole space (see for example
\cite{GimTrutnau2017,OshimaUemura2016}).

Conservativeness of stochastic processes is also studied extensively
on manifolds. The closest conditions in literature to ensure the conservativeness
(also called stochastic completeness) for unbounded $b$ are those
on the symmetric tensor $\nabla^{s}b$ (Ricci curvature or Bakry-Émery
condition) while $\nabla\cdot b$ is the trace of $\nabla^{s}b$.
Hence our condition impose a constrain on the ``scalar curvature''
of the operator.

In this paper, in addition to divergence-free, we only make weak integrability
assumptions on $b$. For divergence-free $b\in L^{2}(\rn)\cap L^{q}(\rn)$
with $q>\frac{n}{2}$, we establish the existence of the Markov semi-group
associated with parabolic equation (\ref{eq: problem equation 1})
using approximation by smooth coefficients, which is inspired by the
ideas from \cite{zhikov2004remarks,Zhikov-2006}. Then, using the
conservativeness implied by the new Aronson type estimate in Theorem
\ref{thm: 1 upper bound}, we proved the uniqueness of the Markov
semi-group. Moreover, the unique semi-group gives the unique weak
solution to (\ref{eq: problem equation 1}) because we show that any
weak solution is an approximation solution.

Once constructed the semi-group, a natural question is the construction
of the corresponding diffusion process. Because the semi-group in
Theorem \ref{thm: Main result} is only defined on $L^{2}(\rn)$,
without regularity results, $\Gamma(t,x,y)$ is only unique a.e. on
$[0,\infty)\times\rn\times\rn$. Therefore, we can not define the
diffusion process for all initial data $x\in\rn$. By \cite[Theorem 3.5]{stannat1999nonsymmetric}
and \cite[Chapter IV]{stannat1999theory}, Stannat constructed the
diffusion in a weaker sense using the framework of generalized Dirichlet
forms, i.e. the process is defined for any initial data $x\in\rn$
except on a zero capacity set. For more details, we refer to \cite{stannat1999nonsymmetric}
and \cite{stannat1999theory}. The solution has local singularities
on a small set because of the local singularities of $b$. Therefore,
if we further assume $b$ to be more regular locally, then we can
define the diffusion process for any initial data $x\in\rn$. A classical
one is the local Lipschitz condition. Then we have the following result
by Theorem \ref{thm: Main result}. As a criterion for the conservativeness
of the diffusion processes, it is interesting by its own.
\begin{cor}
\label{Corollary: diffusion process}Let $b$ be a divergence-free
vector field in $\rn$ with $n\geq3$ that is locally Lipschitz and
$b\in L^{q}(\rn)\cap L^{2}(\rn)$ where $q>\frac{n}{2}$, then the
strong solution to
\[
dX_{t}=dB_{t}+b(X_{t})dt,\qquad X_{0}=x
\]
is conservative.
\end{cor}

Here the local Lipschitz condition can be replaced by other conditions
that implies the local existence and uniqueness of the strong solution
for any initial data $x\in\rn$, for example $b\in L_{loc}^{p}(\rn)$
for $p>n$ (see \cite{KrylovRockner2005}). Then we also have the
same result as stated in Corollary \ref{Corollary: diffusion process}.
It should be mentioned that for time-dependent super-critical $b$
with upper bound on $-\divg b$, Zhang and Zhao\cite{zhang2020} recently
proved the existence and weak uniqueness of the martingale solution
$\mathbb{P}_{s,x}$ for almost all $(s,x)$ except on a null set,
using a new maximum estimate proved in their work. In the special
case when $\divg b=0$, their result implies the conservative of the
solution.

This paper is organized as follows. In section 2, we prove the existence
of weak solution for $b\in L^{2}(\rn)$. In section 3, we give the
proof to Theorem \ref{thm: Main result} where the drift vector field
$b$ belongs to the function space $L^{2}(\rn)\cap L^{q}(\rn)$, $q>\frac{n}{2}$.

\section{Existence of weak solutions}

Firstly, we show the existence of weak solutions to (\ref{eq: problem equation 1})
when $(a,b)$ satisfies conditions (\ref{eq: 1. ellipticity}) and
(\ref{eq: 1. divergence free}) and $b\in L^{2}(\rn)$. A similar
result was proved in \cite{Zhikov-2006} for (\ref{eq: zhikov problem})
and here we borrow the same idea to deal with our case. Here we denote
by $\Gamma(t,x,\xi)$ the fundamental solution to (\ref{eq: problem equation 1}).
\begin{defn}
\label{def weal s}A function $u\in L^{\infty}(0,T;L^{2}(\rn))\cap L^{2}(0,T;H^{1}(\rn))$
is a weak solution to (\ref{eq: problem equation 1}) corresponding
to $(a,b)$ and initial data $u_{0}$ if 
\begin{align*}
\int_{0}^{T}\int_{\mathbb{R}^{n}}u(t,x)\partial_{t}\varphi(t,x)\;dxdt-\int_{0}^{T}\int_{\mathbb{R}^{n}}\langle a(x)\cdot\nabla u(t,x),\nabla\varphi(t,x)\rangle\;dxdt\\
-\int_{0}^{T}\int_{\mathbb{R}^{n}}\langle b(x),\nabla u(t,x)\rangle\varphi(t,x)\;dxdt=-\int_{\rn}u_{0}(x)\varphi(0,x)\ dx
\end{align*}
for any $\varphi\in C_{0}^{\infty}([0,T)\times\rn)$.
\end{defn}

When $b\in L^{q}(\rn)$ with $q\geq n$, for any initial data $u_{0}\in L^{2}(\rn)$,
there exists a unique weak solution $u$ satisfying that $\partial_{t}u\in L^{2}(0,T;H^{-1}(\rn))$
and $u\in C([0,T],L^{2}(\rn))$. Moreover, it satisfies the energy
identity

\begin{equation}
\frac{1}{2}\Vert u(T)\Vert_{2}^{2}+\int_{0}^{T}\int_{\mathbb{R}^{n}}\langle a(x)\cdot\nabla u(t,x),\nabla u(t,x)\rangle\;dxdt=\frac{1}{2}\Vert u_{0}\Vert_{2}^{2}.
\end{equation}
For more details, we refer to \cite{ladyzhenskaia1988linear}. When
$n\geq3$, $b\in L^{2}(\rn)$ we have the following results assuming
that $b$ is divergence-free.
\begin{prop}
Suppose conditions (\ref{eq: 1. ellipticity}) and (\ref{eq: 1. divergence free})
are satisfied and $b\in L^{2}(\rn)$, there exists a weak solution
to (\ref{eq: problem equation 1}) with initial data $u_{0}\in L^{2}(\rn)$.
\end{prop}

\begin{proof}
Denote by $u_{k}$ the weak solution corresponding to $(a,b_{k})$
with the same initial data $u_{0}$ as in Definition \ref{def weal s},
where $b_{k}\in C_{0}^{\infty}(\rn)$ are divergence-free and $b_{k}\rightarrow b$
in $L^{2}(\rn)$. Then $\{u_{k}\}$ is uniformly bounded in $L^{\infty}(0,T;L^{2}(\rn))\cap L^{2}(0,T;H^{1}(\rn))$
and has a sub-sequence which converges weakly to some $u$. This weak
convergence allows us to take $k\rightarrow\infty$ for
\begin{align*}
\int_{0}^{T}\int_{\mathbb{R}^{n}}u_{k}(t,x)\partial_{t}\varphi(t,x)\;dxdt-\int_{0}^{T}\int_{\mathbb{R}^{n}}\langle a(x)\cdot\nabla u_{k}(t,x),\nabla\varphi(t,x)\rangle\;dxdt\\
-\int_{0}^{T}\int_{\mathbb{R}^{n}}\langle b_{k}(x),\nabla u_{k}(t,x)\rangle\varphi(t,x)\;dxdt=-\int_{\rn}u_{0}(x)\varphi(0,x)\ dx
\end{align*}
to obtain that
\begin{align*}
\int_{0}^{T}\int_{\mathbb{R}^{n}}u(t,x)\partial_{t}\varphi(t,x)\;dxdt-\int_{0}^{T}\int_{\mathbb{R}^{n}}\langle a(x)\cdot\nabla u(t,x),\nabla\varphi(t,x)\rangle\;dxdt\\
-\int_{0}^{T}\int_{\mathbb{R}^{n}}\langle b(x),\nabla u(t,x)\rangle\varphi(t,x)\;dxdt=-\int_{\rn}u_{0}(x)\varphi(0,x)\ dx & .
\end{align*}
Hence the limit $u$ is a weak solution to the corresponding problem
with data $(a,b)$.
\end{proof}
We call the weak solution constructed in this way an approximation
solution. Next, we show that every weak solution is an approximation
solution in a weaker sense. This result follows from a similar argument
in \cite{zhang2004strong}.
\begin{prop}
\label{prop: approximation 2}Suppose $b\in L^{2}(\rn)$ and $b_{k}\in C_{0}^{\infty}(\rn)$
are divergence-free such that $b_{k}\rightarrow b$ in $L^{2}(\rn)$.
Let $u$ and $\{u_{k}\}$ be the weak solutions to (\ref{eq: problem equation 1})
on $[0,T]\times\rn$ with the same initial data $u_{0}$, the same
diffusion coefficient $a$, and drifts $b$ and $\{b_{k}\}$ respectively.
Then $u$ is the $L^{\infty}(0,T;L^{1}(\rn))$ limit of functions
$\{u_{k}\}$.
\end{prop}

\begin{proof}
Consider the Cauchy problem 
\[
\partial_{t}u_{k}-\divg(a\cdot\nabla u_{k})+b_{k}\cdot\nabla u_{k}=0
\]
with initial data $u_{k}(x,0)=u_{0}(x)$. Clearly $u_{k}-u$ is a
weak solution to 
\[
\partial_{t}(u_{k}-u)-\divg(a\cdot\nabla(u_{k}-u))+b_{k}\cdot\nabla(u_{k}-u)=(b-b_{k})\cdot\nabla u
\]
with $0$ as the initial value. By assumption, $\Vert(b-b_{k})\cdot\nabla u\Vert_{L^{2}(0,T;L^{1}(\rn))}\rightarrow0$
as $k\rightarrow\infty$. Since $b_{k}\in C_{0}^{\infty}(\rn)$, we
have a representation given by
\[
(u_{k}-u)(t,x)=\int_{0}^{t}\int_{\rn}\Gamma_{k}(t-\tau,x,\xi)(b-b_{k})\cdot\nabla u(\xi,\tau)\ d\xi d\tau,
\]
where $\Gamma_{k}$ is the fundamental solution corresponding to $b_{k}$.
Then $\Gamma_{k}^{*}(t,\xi,x):=\Gamma_{k}(t,x,\xi)$ is the fundamental
solution to $(\partial_{t}-L_{k}^{*})u=0$, which is of the same form
as the original equation (\ref{eq: problem equation 1}) up to a sign
on the drift. Hence
\begin{equation}
\int_{\rn}\Gamma_{k}(t-\tau,x,\xi)\ dx=1\label{eq: conservative of the dual}
\end{equation}
for any fixed $(t-\tau,\xi)$. This implies that
\[
\int_{\rn}\vert u_{k}-u\vert(t,x)\ dx\leq\int_{0}^{t}\int_{\rn}\vert b-b_{k}\vert\vert\nabla u\vert\ d\xi d\tau\rightarrow0
\]
and the proof is done.
\end{proof}
The proposition above implies that all weak solutions are approximation
solutions. Here the divergence-free condition is the key to have the
dual operator being conservative to obtain (\ref{eq: conservative of the dual}).

\section{Uniqueness of the approximation semi-group and its kernel}

In this section we prove our main result Theorem \ref{thm: Main result}.
The idea is to construct a unique approximation Markov semi-group
corresponding to generator $L=\divg(a\cdot\nabla)-b\cdot\nabla$.
Since $a$ is only Borel measurable, the generator $L$ is not well
defined as a differential operator. Hence we will construct $L$ in
the following, while we still use formal expression $L=\divg(a\cdot\nabla)-b\cdot\nabla$
for simplicity of notation. We start with the bi-linear form 
\[
\dir(u,v)=\int_{\rn}\langle\nabla u,a\cdot\nabla v\rangle+(b\cdot\nabla u)v\ dx.
\]
Naturally we consider the elliptic problem and its weak solutions,
which is standard in literature.
\begin{defn}
Let $(a,b)$ satisfies (\ref{eq: 1. ellipticity}), (\ref{eq: 1. divergence free})
and $b\in L^{2}(\rn)$. For $f\in L^{2}(\rn)$, if there exists a
$u\in H^{1}(\rn)$ such that 
\[
\int_{\rn}\langle\nabla u,a\cdot\nabla\varphi\rangle+(b\cdot\nabla u)\varphi+\alpha u\varphi\ dx=\int_{\rn}f\varphi\ dx
\]
for all $\varphi\in C_{0}^{\infty}(\rn)$, we call $u$ a weak solution
to the elliptic problem $(\alpha-L,f)$, where $\alpha\geq0$.
\end{defn}

For $b\in C_{0}^{\infty}(\rn)$, the bi-linear form is actually a
Dirichlet form. We recall the following result on Dirichlet forms
in \cite[Chapter 1]{ma2012introduction}.
\begin{thm}
\label{thm: Dirichlet form}Suppose $(a,b)$ satisfies (\ref{eq: 1. ellipticity}),
(\ref{eq: 1. divergence free}) and $b\in C_{0}^{\infty}(\rn)$, then
$\left(\mathcal{E},H^{1}(\mathbb{R}^{n})\right)$, where 
\[
\dir(u,v)=\int_{\rn}\langle\nabla u,a\cdot\nabla v\rangle+(b\cdot\nabla u)v\ dx
\]
with $u,v\in H^{1}(\rn)$, is a (non-symmetric) Dirichlet form. We
still use $L$ together with its domain $D(L)$ to denote the generator
associated with the Dirichlet form $\left(\mathcal{E},H^{1}(\mathbb{R}^{n})\right)$.
The resolvent $R_{\alpha}=(\alpha-L)^{-1}$ for $\alpha>0$ is a bounded
linear operator from $L^{2}(\rn)$ to $L^{2}(\rn)$ with $\Vert(\alpha-L)^{-1}\Vert_{L^{2}\rightarrow L^{2}}\leq\alpha^{-1}$,
and it satisfies that
\begin{equation}
\dir(R_{\alpha}f,v)+\alpha\int_{\rn}(R_{\alpha}f)v\ dx=\int_{\rn}fv\ dx.\label{eq: resolvent bilinear form}
\end{equation}
\end{thm}

Thus for $b\in C_{0}^{\infty}(\rn)$, $\divg(a\cdot\nabla)-b\cdot\nabla$
is understood as the generator $L$ defined as in Theorem \ref{thm: Dirichlet form}
above. Clearly, for any $f\in L^{2}(\rn)$, $(\alpha-L)^{-1}f$ is
the unique weak solution to $(\alpha-L,f)$. We can take $v=(\alpha-L)^{-1}f$
and derive that 
\begin{equation}
\Vert(\alpha-L)^{-1}f\Vert_{H^{1}}\leq\frac{1}{\min\{\lambda,\alpha\}}\Vert f\Vert_{L^{2}}\qquad\mbox{and }\qquad\Vert(\alpha-L)^{-1}f\Vert_{L^{2}}\leq\frac{1}{\alpha}\Vert f\Vert_{L^{2}}\label{eq: elliptic estimate}
\end{equation}
for all $\alpha>0$ and $f\in L^{2}(\rn)$. Next we prove the following
estimate on $R_{\alpha}$, which follows from \cite{Zhikov-2006},
plays an important role in proving our main result.
\begin{lem}
\label{lem: 3. elliptic case tail estimate}Suppose $b\in C_{0}^{\infty}(\rn)$
and $L$ as in Theorem \ref{thm: Dirichlet form}, set $u=(1-L)^{-1}f$
for $f\in C_{0}^{\infty}(\rn)$. Then for $n\geq3$, we have 
\[
\int_{\rn}\left[\ln(\vert x\vert^{2}+e)\right]^{2\gamma}u^{2}(x)\ dx\leq C_{0}\int_{\rn}\left[\ln(\vert x\vert^{2}+e)\right]^{2\gamma}f^{2}(x)\ dx
\]
with sufficiently small positive $\gamma$ and constant $C_{0}$ depending
only on $n$, $\lambda$, $\gamma$ and $\Vert b\Vert_{L^{q}(\rn)}$
with $q>\frac{n}{2}$.
\end{lem}

\begin{proof}
Let $\psi=\gamma\psi_{0}$, $\psi_{0}=\ln\ln(\vert x\vert^{2}+e)$,
for $\gamma>0$, and consider the operator $L_{\psi}=e^{\psi}Le^{-\psi}.$
For $v=e^{\psi}u$, we have $L_{\psi}v-v=g=e^{\psi}f$ and 
\[
\int_{\rn}-\langle\nabla(e^{\psi}v),a\cdot\nabla(e^{-\psi}v)\rangle-b\cdot\nabla(e^{-\psi}v)e^{\psi}v-v^{2}\ dx=\int_{\rn}gv\ dx.
\]
It follows, together with (\ref{eq: 1. ellipticity}) and (\ref{eq: 1. divergence free}),
that
\[
\int_{\rn}\lambda\vert\nabla v\vert^{2}-\frac{1}{\lambda}\gamma^{2}\vert\nabla\psi_{0}\vert^{2}v^{2}-\gamma(b\cdot\nabla\psi_{0})v^{2}+v^{2}\ dx\leq-\int_{\rn}gv\ dx.
\]
Notice that 
\[
\vert\nabla\psi_{0}\vert\leq\frac{2\vert x\vert}{(\vert x\vert^{2}+e)\ln(\vert x\vert^{2}+e)},
\]
which is bounded. Hence we have 
\begin{align*}
\int_{\rn}(b\cdot\nabla\psi_{0})v^{2}\ dx & \leq C\Vert b\Vert_{L^{q}}\Vert\nabla\psi_{0}\Vert_{L^{\infty}}\Vert v\Vert_{L^{2}}^{1-\theta}\Vert\nabla v\Vert_{L^{2}}^{1+\theta}\\
 & \leq C\Vert b\Vert_{L^{q}}\Vert\nabla\psi_{0}\Vert_{L^{\infty}}C(\theta)\left(\Vert v\Vert_{L^{2}}^{2}+\Vert\nabla v\Vert_{L^{2}}^{2}\right)
\end{align*}
where $\theta=\frac{n}{q}-1$ and $C$ depends on $n,q$. Now we can
take $\gamma$ small enough such that $\Vert v\Vert_{L^{2}}\leq C_{0}\Vert g\Vert_{L^{2}}$
and the proof is complete.
\end{proof}
Given divergence-free $b\in L^{q}(\rn)\cap L^{2}(\rn)$ and a sequence
of smooth functions $b_{k}\rightarrow b$ in $L^{q}(\rn)\cap L^{2}(\rn)$,
Lemma \ref{lem: 3. elliptic case tail estimate} implies that for
each fixed $f\in C_{0}^{\infty}(\rn)$, $L^{2}$-norm of the approximate
sequence $\left\{ (1-L_{k})^{-1}f\right\} $ will be uniformly concentrated
on some balls. Now we can prove the compactness of $\left\{ (1-L_{k})^{-1}f\right\} $.
\begin{lem}
\label{lem: 3. compactness lemma}Given divergence-free $b\in L^{q}(\rn)\cap L^{2}(\rn)$,
smooth approximations $b_{k}\rightarrow b$ in space $L^{q}(\rn)\cap L^{2}(\rn)$,
and $f\in L^{2}(\rn)$, the sequence $\{(1-L_{k})^{-1}f\}$ is strongly
compact in $L^{2}(\rn)$ and weakly compact in $H^{1}(\rn)$.
\end{lem}

\begin{proof}
By $\Vert(1-L_{k})^{-1}f\Vert_{H^{1}}\leq\frac{1}{\min\{\lambda,1\}}\Vert f\Vert_{L^{2}}$,
we have that the sequence $\left\{ (1-L_{k})^{-1}f\right\} $ is weakly
compact in $H^{1}(\rn)$. To prove the strong compactness in $L^{2}(\rn)$,
recall that we have proved $\Vert(1-L_{k})^{-1}\Vert_{L^{2}\rightarrow L^{2}}\leq1$
for all $k$. Since the convergence of bounded linear operators is
determined by its convergence on a dense subset (see Theorem 6 in
\cite[Ch15]{lax2002functional}), it is sufficient to establish the
compactness of $\{(1-L_{k})^{-1}f\}$ for $f$ in a dense subset of
$L^{2}(\rn)$. For $f\in C_{0}^{\infty}(\rn)$, by Lemma \ref{lem: 3. elliptic case tail estimate}
and inequality $\Vert(1-L_{k})^{-1}f\Vert_{H^{1}}\leq\frac{1}{\min\{\lambda,1\}}\Vert f\Vert_{L^{2}}$,
the compactness of $\{(1-L_{k})^{-1}f\}$ in $L^{2}(\rn)$ follows
from the Fréchet–Kolmogorov theorem \cite[Chapter X, Section 1]{yosida1980functional}.
\end{proof}
Lemma \ref{lem: 3. compactness lemma} above allows us to take $k\rightarrow\infty$
and define the generator $L$ for singular $b$ as the limit of $L_{k}$.
\begin{lem}
\label{lem: convergence of resolvent}Given $L_{k}$ defined as in
Theorem \ref{thm: Dirichlet form} corresponding to $b_{k}$ which
converges to $b$ in $L^{q}(\rn)\cap L^{2}(\rn)$, after a possible
selection of a sub-sequence (denoted as $L_{k}$ again), there exists
a closed operator $L$ defined on a dense subset of $L^{2}(\rn)$
such that $\Vert(\alpha-L)^{-1}\Vert_{L^{2}\rightarrow L^{2}}\leq\alpha^{-1}$
for all $\alpha>0$ and
\[
(\alpha-L_{k})^{-1}f\rightarrow(\alpha-L)^{-1}f\qquad\mbox{ in }L^{2}(\rn)
\]
 for all $\alpha>0$ and $f\in L^{2}(\rn)$.
\end{lem}

\begin{proof}
We first consider the case when $\alpha=1$. By Theorem 6 in \cite[Ch15]{lax2002functional},
convergence of bounded linear operators $\left\{ (1-L_{k})^{-1}\right\} _{k=1}^{\infty}$
is determined by its convergence on a dense subset of $L^{2}(\rn)$.
We apply Lemma \ref{lem: 3. compactness lemma} to $f$ in a countable
dense subset of $L^{2}(\rn)$. Then using the diagonal argument, we
can find a sub-sequence of $\left\{ (1-L_{k})^{-1}\right\} _{k=1}^{\infty}$
that converges strongly. We still denote the sub-sequence as $(1-L_{k})^{-1}$
and denote its limit as $S$, i.e. 
\[
(1-L_{k})^{-1}f\rightarrow Sf
\]
strongly in $L^{2}(\rn)$ for $f\in L^{2}(\rn)$.

Given any $f\in L^{2}(\rn)$, since $\{(1-L_{k})^{-1}f\}$ is weakly
compact in $H^{1}$, it also converges to $Sf$ weakly in $H^{1}$.
It is easy to see that the limit $Sf$ is a weak solution to $(1-L,f)$.
Since $S$ is bounded linear operator from $L^{2}(\rn)$ to itself,
we can define its adjoint operator $S^{*}$ by $\langle Sf,g\rangle=\langle f,S^{*}g\rangle$
for all $f,g\in L^{2}(\rn)$. We already know that $\lim_{k\rightarrow\infty}\langle(1-L_{k})^{-1}f,g\rangle=\langle Sf,g\rangle$
for all $f,g\in L^{2}(\rn)$ and $\langle(1-L_{k})^{-1}f,g\rangle=\langle f,(1-L_{k}^{*})^{-1}g\rangle$.
Hence we can see that $S^{*}g$ is a weak solution to $(1-L^{*},g)$.
Proposition \ref{prop: uniqueness} below implies that both $S$ and
$S^{*}$ have $K(S)=K(S^{*})=0$ and hence have dense range in $L^{2}(\rn)$
by $K(S^{*})=R(S)^{\perp}$. Now we can define $L=1-S^{-1}$, which
has dense domain $D(L)$ and $D(L)\subset H^{1}$. Since $S^{-1}$
is a closed operator, $L$ is also closed.

Clearly, for each $u\in D(L)$, it is the weak solution to $(-L,-Lu)$.
Hence for any $\alpha>0$, $(\alpha-L)$ is also a closed operator
and we can define $(\alpha-L)^{-1}f$ to be the weak solution to $(\alpha-L,f)$
for $f$ in the range of $(\alpha-L)$, i.e. $f\in R(\alpha-L)$.
Notice that for $(\alpha-L)$ and its dual operator $(\alpha-L^{\ast})$,
Proposition \ref{prop: uniqueness} implies that $K(\alpha-L^{\ast})=0$.
By the closed range theorem, we have that $R(\alpha-L)=K(\alpha-L^{\ast})^{\perp}=L^{2}(\rn)$.
Hence the resolvent operator $(\alpha-L)^{-1}$ is well defined for
any $\alpha>0$. Finally we prove that $(\alpha-L_{k})^{-1}f\rightarrow(\alpha-L)^{-1}f$
for any $\alpha>0$ and $f\in L^{2}(\rn)$. As shown in Theorem 1.3
in \cite[Ch.8]{kato2013perturbation}, we can derive
\begin{align*}
 & (\alpha-L_{k})^{-1}-(\alpha-L)^{-1}\\
 & \qquad\qquad=\left(1+(\alpha-1)(\alpha-L_{k})^{-1}\right)\left((1-L_{k})^{-1}-(1-L)^{-1}\right)\left(1+(\alpha-1)(\alpha-L)^{-1}\right),
\end{align*}
from the resolvent equation and $(1-L_{k})^{-1}f\rightarrow(1-L)^{-1}f$
implies that $(\alpha-L_{k})^{-1}f\rightarrow(\alpha-L)^{-1}f$ for
any $\alpha>0$. Since estimates (\ref{eq: elliptic estimate}) are
true for $\{(\alpha-L_{k})^{-1}f\}$, it is also true for the limit
$(\alpha-L)^{-1}f.$
\end{proof}
\begin{rem}
\label{rem: Unique} Given $f\in L^{2}(\rn)$, we have that $Sf$
is the limit of $\{(1-L_{k})^{-1}f\}$ weakly in $H^{1}(\rn)$ and
it is easy to check that $Sf$ is a weak solution to $(1-L,f)$. Next
we show that for $b\in L^{q}(\rn)\cap L^{2}(\rn)$, there is a unique
$S$ defined as in last Theorem. The uniqueness of $S$ implies that
the definition of $L$ is independent of the choice of the convergent
sub-sequence.
\end{rem}

\begin{prop}
\label{prop: uniqueness}Suppose $(a,b)$ satisfies conditions (\ref{eq: 1. ellipticity})
and (\ref{eq: 1. divergence free}). For any $f\in L^{2}$, there
exists a unique weak solution $u\in H^{1}$ to the elliptic problem
$(\alpha-L,f)$ for $n\geq3$, $b\in L^{q}(\rn)\cap L^{2}(\rn)$ and
$\alpha>0$.
\end{prop}

\begin{proof}
The existence of weak solution can be obtained using the same approximation
argument as in Lemma \ref{lem: convergence of resolvent}. To prove
the uniqueness, we consider a weak solution $u$ to $(\alpha-L,0)$.
Here we can take the test function to be $h=\bar{u}\varphi$ with
$\bar{u}=u\wedge N\vee(-N)$ and $\varphi\in C_{0}^{\infty}$. Given
any $r>0$, there is $\varphi_{r}$ satisfying 
\[
\varphi_{r}=\begin{cases}
1 & \vert x\vert\leq\frac{r}{2}\\
0 & \vert x\vert\geq r
\end{cases},\qquad\vert\nabla\varphi\vert\leq\frac{4}{r}.
\]
Then we have 
\[
\int_{\rn}\langle\nabla u,a\cdot\nabla(\bar{u}\varphi_{r})\rangle+b\cdot\nabla u(\bar{u}\varphi_{r})+\alpha u(\bar{u}\varphi_{r})\ dx=0.
\]
Because $\bar{u}\varphi_{r}\rightarrow\bar{u}$ in $H^{1}(\rn)$ and
a.e., we can take $r\rightarrow\infty$ to obtain 
\[
\int_{\rn}\langle\nabla u,a\cdot\nabla(\bar{u})\rangle+b\cdot\nabla u(\bar{u})+\alpha u(\bar{u})\ dx=0.
\]
Next we consider the second term in the equation above. Since $\int_{\rn}b\cdot\nabla\bar{u}\bar{u}\ dx=0$,
we have 
\begin{align*}
\int_{\rn}b\cdot\nabla u\bar{u}\ dx & =\int_{\rn}b\cdot(\nabla u-\nabla\bar{u})\bar{u}\ dx\\
 & =N\int_{\{u>N\}}b\cdot(\nabla u-\nabla\bar{u})\ dx-N\int_{\{u<-N\}}b\cdot(\nabla u-\nabla\bar{u})\ dx=0
\end{align*}
Now we can take $N\rightarrow\infty$ to obtain 
\[
\int_{\rn}\langle\nabla u,a\cdot\nabla u\rangle+\alpha u^{2}\ dx=0
\]
and conclude that $u=0$.
\end{proof}
Finally, to prove the representation (\ref{eq: representation of the solution}),
we also need the convergence of the fundamental solutions.
\begin{prop}
\label{prop: tightness of finite dimensional}Given a sequence of
probability measures $\left\{ P_{n}\right\} $ on $\rn$ which have
densities $\left\{ f_{n}\right\} $ uniformly bounded from above by
a continuous function $h$. Suppose $h$ satisfies
\[
\lim_{R\rightarrow\infty}\int_{B(0,R)^{c}}h(x)\ dx=0,
\]
where $B(0,R)$ is the open ball in $\rn$ centered at $0$ with radius
$R$. Then $\left\{ P_{n}\right\} $ is weakly compact in the space
of probability measure. Suppose we take a convergent sub-sequence,
then its limit $P$ has density $f$ which is also bounded from above
by $h$.
\end{prop}

\begin{proof}
It is easy to see that $\left\{ P_{n}\right\} $ is tight, which implies
that it is weakly compact by Prohorov's theorem. So we just need to
show that $P$ has density $f$ which is bounded by $h$. Firstly,
we show that $P$ is absolutely continuous with respect to the Lebesgue
measure $m$. Suppose $A\subset\rn$ such that $m(A)=0$, then there
is a decreasing sequence of open sets $\{O_{i}\}$ containing $A$
such that $\lim_{i\rightarrow\infty}m(O_{i})=0$. Therefore $\lim_{i\rightarrow\infty}P_{n}(O_{i})\rightarrow0$
uniformly for all $P_{n}$. By the Portmanteau theorem \cite[Theorem 1.1.1]{stroock2007multidimensional},
we have $P(O_{i})\leq\limsup_{n\rightarrow\infty}P_{n}(O_{i})$, which
implies that $\lim_{i\rightarrow\infty}P(O_{i})=0$ and hence $P(A)=0$.
So $P$ has a density $f$ by Radon–Nikodym's theorem.

Next we show that this $f$ is bounded by $h$. If not, we can find
a bounded set $A$ such that $m(A)>0$ and $f>h$ a.e. on $A$. Since
$h$ is continuous, we can find an open set $O$ small enough such
that it contains $A$ and $P(O)>\int_{O}h\geq P_{n}(O)$ for all $n$.
Clearly this contradicts to that $P_{n}\rightarrow P$ weakly in measure.
\end{proof}
Now we are in a position to complete the proof of Theorem \ref{thm: Main result}.
\begin{proof}
By the fundamental approximation theorem of semi-groups in \cite[Cp 9, Theorem 2.16]{kato2013perturbation},
the convergence of resolvents proved in Lemma \ref{lem: convergence of resolvent}
implies that $e^{tL_{k}}\rightarrow e^{tL}$ as bounded linear operators
from $L^{2}(\rn)$ to $L^{2}(\rn)$ and is uniform for $t$ in any
finite interval $[0,T]$. Moreover, we know that $L$ is unique by
Remark \ref{rem: Unique} and Proposition \ref{prop: uniqueness}.
Using Proposition \ref{prop: approximation 2}, we know that any weak
solution to (\ref{eq: problem equation 1}) with initial condition
$u_{0}\in L^{2}(\mathbb{R}^{n})$ must equal to $e^{tL}u_{0}$. Therefore,
$e^{tL}$ is the unique semi-group which generates the unique weak
solution.

Next, we show the existence of the unique fundamental solution corresponding
to (\ref{eq: problem equation 1}). We have proved above that $u_{k}\rightarrow e^{tL}u_{0}$
in $L^{2}(\rn)$, which implies that there is a subsequence, denoted
as $u_{k}$ again, converging a.e. to $e^{tL}u_{0}$. Let $\Gamma_{k}(t,x,y)$
be the corresponding fundamental solution to $(\partial_{t}-L_{k})u=0$.
Then
\[
u_{k}(t,x)=\int_{\rn}\Gamma_{k}(t,x,y)u_{0}(y)\ dy=e^{tL_{k}}u_{0}
\]
for any $u_{0}\in L^{2}(\rn)$ and $k=1,2,\cdots$. By Theorem \ref{thm: 1 upper bound}
and Proposition \ref{prop: tightness of finite dimensional}, we have
that for each fixed $(t,x)$ (and $(t,y)$), the family of transition
probabilities $\left\{ \Gamma_{k}(t,x,y)\ dy\right\} $ (and also
the family$\left\{ \Gamma_{k}(t,x,y)\ dx\right\} $) is tight and
hence converges weakly in measure to some $\Gamma(t,x,y)\ dy$ which
has the same upper bound as that of $\Gamma_{k}(t,x,y)$. Define 
\[
u(t,x)=\int_{\rn}\Gamma(t,x,y)u_{0}(y)\ dy
\]
for $u_{0}\in C_{0}^{\infty}(\rn)$, then for each fixed $(t,x)$,
$u_{k}(t,x)$ has a subsequence that converges to $u(t,x)$ by the
weak convergence of measure. Now we proved that $u=e^{tL}u_{0}$ in
$L^{2}(\rn)$. Since $C_{0}^{\infty}(\rn)$ is dense in $L^{2}(\rn)$,
we can extend it to $L^{2}(\rn)$ and conclude that operator $e^{tL}$
has a kernel $\Gamma(t,x,y)$.
\end{proof}
\bibliographystyle{plain}
\bibliography{L2semigroup}

\end{document}